\documentclass[11pt]{amsart}

\usepackage{amsmath}
\usepackage{amssymb}
\usepackage{amsfonts}

\usepackage{amsthm}
\usepackage{enumerate}

\usepackage[mathscr]{eucal}
\usepackage{mathrsfs}

\usepackage{bbm}

\usepackage{subfigure}
\usepackage{tikz}
\usetikzlibrary{patterns}

\usepackage{graphicx}
\usepackage{verbatim}

\newcommand{\Be}{\begin{equation}}
\newcommand{\Ee}{\end{equation}}

\newcommand{\Ba}[1]{\begin{array}{#1}}
\newcommand{\Ea}{\end{array}}
\newcommand{\Bea}{\begin{eqnarray}}
\newcommand{\Eea}{\end{eqnarray}}
\newcommand{\Beas}{\begin{eqnarray*}}
\newcommand{\Eeas}{\end{eqnarray*}}
\newcommand{\Benu}{\begin{enumerate}}
\newcommand{\Eenu}{\end{enumerate}}
\newcommand{\Bi}{\begin{itemize}}
\newcommand{\Ei}{\end{itemize}}

\def\intslash{\rlap{\kern  .32em $\mspace {.5mu}\backslash$ }\int}
\def\qsl{{\rlap{\kern  .32em $\mspace {.5mu}\backslash$ }\int_{Q_x}}}

\newcommand {\Span} {\operatorname{span}}

\def\emph#1{{\it #1 }}

\def\inn#1#2{\langle#1,#2\rangle}

\def\card{\text{\rm card}}
\def\lc{\lesssim}

\def\eps{\varepsilon}

\def\bbD{{\mathbb {D}}}
\def\bbE{{\mathbb {E}}}

\def\bbN{{\mathbb {N}}}

\def\bbR{{\mathbb {R}}}

\def\bbZ{{\mathbb {Z}}}

\def\sH{{\mathscr {H}}}

\def\cI{{\mathcal {I}}}

\def\cT{{\mathcal {T}}}

\def\cV{{\mathcal {V}}}

 at 10 true pt

\def\T{{\hbox{\bf T}}}

\def\be#1{\begin{equation}\label{ #1}}
\def\endeq{\end{equation}}
\def\endal{\end{align}}
\def\bas{\begin{align*}}
\def\eas{\end{align*}}
\def\bi{\begin{itemize}}
\def\ei{\end{itemize}}

\def\eps{\varepsilon}

\def\emph#1{{\it #1}}
\def\textbf#1{{\bf #1}}

\def\bbone{{\mathbbm 1}}

\theoremstyle{plain}
  \newtheorem{theorem}{Theorem}[section]

   \newtheorem{proposition}[theorem]{Proposition}

\theoremstyle{remark}
   \newtheorem{remark}[theorem]{Remark}

\theoremstyle{definition}

\numberwithin{equation}{section}
\newcommand {\SC} {{\mathbb C}}
\newcommand {\SD} {{\mathbb D}}
\newcommand {\SE} {{\mathbb E}}

\newcommand {\SR} {{\mathbb R}}

\newcommand {\SZ} {{\mathbb Z}}

\newcommand {\e} {{\varepsilon}}

\newcommand {\mand} {{\quad\mbox{and}\quad}}
\renewcommand {\mid} {{\,\,\,\colon\,\,\,}}

\def\lan#1#2{{\langle{#1},{#2}\rangle}}

\newcounter{reg}
\def\XXint#1#2#3{{\setbox0=\hbox{$#1{#2#3}{\int}$ }
\vcenter{\hbox{$#2#3$ }}\kern-.6\wd0}}

\definecolor{ascol}{rgb}{0,0,1.} 
\definecolor{ggcol}{cmyk}{.74, 0, 1, .41} 
\definecolor{tucol}{rgb}{0.9,.5,0} 

\begin{document}

\title
[Haar multiplers]
{A sufficient condition for Haar multipliers in Triebel-Lizorkin spaces}

\author[G. Garrig\'os \ \ \ A. Seeger \ \ \ T. Ullrich] {Gustavo Garrig\'os   \ \ \ \   Andreas Seeger \ \ \ \ Tino Ullrich}

\address{Gustavo Garrig\'os\\ Department of Mathematics\\University of Murcia\\30100 Espinardo\\Murcia, Spain} \email{gustavo.garrigos@um.es}

\address{Andreas Seeger \\ Department of Mathematics \\ University of Wisconsin \\480 Lincoln Drive\\ Madison, WI,53706, USA} \email{seeger@math.wisc.edu}
\address{Tino Ullrich\\ Fakult\"at f\"ur Mathematik\\ Technische Universit\"at  Chemnitz\\09107 Chemnitz, Germany}
\email{tino.ullrich@mathematik.tu-chemnitz.de}



\date{\today}

\dedicatory{In memory of Guido Weiss}

\maketitle

\begin{abstract}
We consider Haar multiplier operators $T_m$ acting on Sobolev spaces, and more generally  Triebel-Lizorkin spaces $F^s_{p,q}(\SR)$, 
for indices in which the Haar system is not unconditional.
When $m$ depends only on the Haar frequency, we give a sufficient condition for the boundedness of $T_m$ in $F^s_{p,q}$, 
in  terms of the  variation norms $\|m\|_{V_u}$, which is optimal in $u$ (up to endpoints) when $p, q> 1$. 
\end{abstract}

\section{Introduction}

\noindent 
Consider the classical Haar system in $\SR$, 
\Be\label{HaarS}
\sH=\big\{h_{j,\mu}\mid j\geq-1, \,\mu\in\SZ\big\},
\Ee
where, if $h=\bbone_{[0,1/2)}-\bbone_{[1/2,1)}$, we let
\[
h_{j,\mu}(x)=h(2^jx-\mu)\,,\quad\mbox{for}\quad
j=0,1,2,\ldots, \;\;\mu\in\SZ,
\]
while for $j=-1$ we let
\[
h_{-1,\mu}=
\bbone_{[\mu,\mu+1)},\quad \mu\in\SZ.
\]
We shall refer to the elements of the family
$\sH_j=\{h_{j,\mu}\,:\,\mu\in \bbZ \}$ as {\it Haar functions of frequency} $2^j$.

Let $F^s_{p,q}$ denote the usual Triebel-Lizorkin space in $\SR$; see \cite{Tr83}.
It is known from the work of Triebel \cite[Theorem 2.9.ii]{triebel-bases} that $\sH$ is an unconditional 
basis of $F^s_{p,q}(\bbR)$ when 
$s$ belongs to the range 
\Be\label{eq:uncond-range} \max\Big\{1/p-1,1/q-1\Big\} < s < \min\Big\{1/p,1/q, 1\Big\}.
\Ee 
That this range is actually optimal was shown by the last two authors in \cite{su,sudet}.
More recently, we proved in \cite{gsu} that $\sH$ is a Schauder basis  of $F^s_{p,q}(\SR)$ 
(with respect to natural enumerations) in the larger range
\Be
1/p-1 < s < \min\Big\{1/p, 1\Big\},\quad \mbox{(for all $0<q<\infty$)},
\label{cond_range}
\Ee
while at the endpoints (see \cite{gsu-TL}) the property holds if and only if 
\Be
s=1/p-1\mand 1/2<p\leq 1,
\label{cond_range_end}
\Ee
also for all $0<q<\infty$.
These regions are depicted in Figure  \ref{fig1} below.

\begin{figure}[h]
 \centering
\subfigure
{\begin{tikzpicture}[scale=2]

\node [right] at (0.75,-0.5) {{\footnotesize unconditional}};

\draw[->] (-0.1,0.0) -- (2.1,0.0) node[right] {$\frac{1}{p}$};
\draw[->] (0.0,-0.0) -- (0.0,1.1) node[above] {$s$};
\draw (0.0,-1.1) -- (0.0,-1.0)  ;

\draw (1.0,0.03) -- (1.0,-0.03) node [below] {$1$};
\draw (2.0,0.03) -- (2.0,-0.03) node [below] {$2$};
\draw (0.03,1.0) -- (-0.03,1.00);
\node [left] at (0,0.9) {$1$};
\draw (0.03,.5) -- (-0.03,.5) node [left] {$\tfrac{1}{q}$};
\draw (0.03,-.5) -- (-0.03,-.5) node [left] {$\tfrac{1}{q}${\small{$-1$}}};
\draw (0.03,-1.0) -- (-0.03,-1.00) node [left] {$-1$};

\draw[dotted] (1.0,0.0) -- (1.0,1.0);
\draw[dotted] (0,1.0) -- (1.0,1.0);
\draw[dotted] (2,0.0) -- (2,1.0);

\path[fill=green!70, opacity=0.4] (0.0,0.0) -- (.5,.5)-- (1.5,0.5) -- (1,0)--(.5,-.5) -- (0,-0.5)--(0,0);
\draw[dotted] (0,0.5)--(1.5,0.5);
\draw[dotted] (0,-0.5)--(0.5,-0.5);

\draw[dashed] (0.0,-1.0) -- (0.0,0.0) -- (1.0,1.0) -- (2,1.0) -- (1.0,0.0) --
(0.0,-1.0);

\end{tikzpicture}
}
\subfigure
{
\begin{tikzpicture}[scale=2]

\node [right] at (0.75,-0.5) {{\footnotesize Schauder}};

\draw[->] (-0.1,0.0) -- (2.1,0.0) node[right] {$\frac{1}{p}$};
\draw[->] (0.0,-0.0) -- (0.0,1.1) node[above] {$s$};
\draw (0.0,-1.1) -- (0.0,-1.0)  ;

\draw (1.0,0.03) -- (1.0,-0.03) node [below] {$1$};
\draw (2.0,0.03) -- (2.0,-0.03) node [below] {$2$};
\draw (0.03,1.0) -- (-0.03,1.00);
\node [left] at (0,0.9) {$1$};
\draw (0.03,-1.0) -- (-0.03,-1.00) node [left] {$-1$};

\draw[dotted] (1.0,0.0) -- (1.0,1.0);
\draw[dotted] (0,1.0) -- (1.0,1.0);
\draw[dotted] (2,0.0) -- (2,1.0);

\draw[dashed, thick] (1,0) -- (0.0,-1.0)--(0.0,0.0) -- (1,1)--(2,1);


\draw[white, fill=red!70, opacity=0.4] (0,0) -- (1,1)
-- (2,1.0) -- (1.0,0) --(0,-1)--(0,0);

\draw[very thick,red] (1.98,0.98) -- (1.0,0.0);
\fill[red] (1,0) circle (1pt);
\fill[white] (2,1) circle (1pt);
\draw (2,1) circle (1pt);


\end{tikzpicture}

}
\caption{Parameter domain  for  $\sH$ to be an unconditional basis (left figure) 
or a Schauder basis (right figure) in $F^s_{p,q}(\SR)$. 
}\label{fig1}
\end{figure}
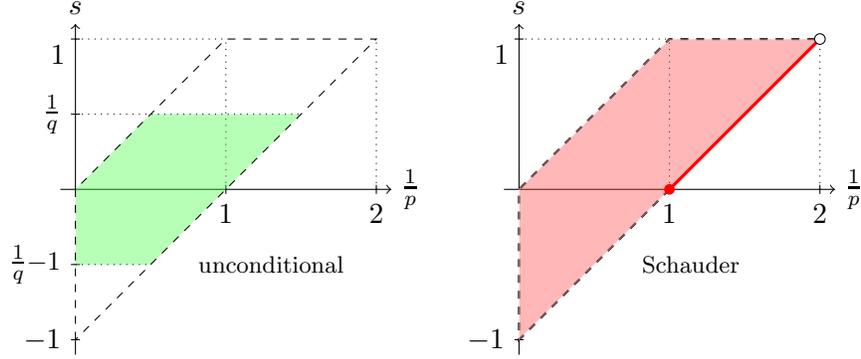

We shall mainly be interested in values of the parameters
outside the region of unconditionality.
In that range, it becomes a natural question to find sufficient conditions on a sequence $\{m_{j,\mu}\}$
so that the mapping
\[
f\longmapsto \sum_{j\geq0}\sum_{\mu\in\SZ}m_{j,\mu}\,2^j\lan f{h_{j,\mu}} h_{j,\mu},
\]
defined say for $f\in\Span\sH$, extends as a bounded linear operator in the space $F^s_{p,q}$.

\

In this paper we regard this problem in the special case when the sequence is constant in each frequency level, namely, if 
$m=\{m(j)\}_{j\geq0}$, we consider the operators
\[ 
T_m f = \sum_{j\ge 0}  m(j)\, \bbD_j f,
\]
where $\SD_j$ denotes the orthogonal projection onto the space generated by $\sH_j$, that is
\[
\SD_jf=\sum_{\mu\in\SZ}2^j\lan f{h_{j,\mu}}\,h_{j,\mu},\quad j\geq0.
\]
It is well known that one can write 
\[
\bbD_j=\bbE_{j+1}-\bbE_{j},
\]
where $\SE_j$ is the conditional expectation operator defined by
\Be
\label{expect}
\bbE_j f(x)=\sum_{\mu \in \bbZ} \bbone_{I_{j,\mu }}(x)\,2^{j} \,\int_{I_{j,\mu }}f(y) dy\,,
\Ee
associated with the dyadic intervals $I_{j,\mu}=[\mu2^{-j},(\mu+1)2^{-j})$, $\mu\in\SZ$.

The uniform boundedness of the operators $\SE_N$ in $F^s_{p,q}$ (and $B^s_{p,q}$) has been throughly studied in the papers
\cite{gsu, gsu-endpt, gsu-TL}. In particular, it is shown in those papers that $\sH$ is a Schauder basis of $F^s_{p,q}$ 
(with respect to natural enumerations) if and only if
\[
\sup_{N\geq0}\big\|\SE_N\big\|_{F^s_{p,q}\to F^s_{p,q}}<\infty\mand \mbox{$\Span\sH$ is dense in $F^s_{p,q}$,} 
\]
and this in turn is equivalent to $(s,p,q)$ belonging to the ranges in \eqref{cond_range} and \eqref{cond_range_end}.
In those cases, an elementary summation by parts argument and the $\sigma$-triangle inequality, with $\sigma=\min\{1,p,q\}$, imply that 
\Be
\|T_mf\|_{F^s_{p,q}}\lesssim \|m\|_{\ell^\infty}+\Big[\sum_{j=1}^\infty|m(j)-m(j-1)|^\sigma\Big]^{\frac1\sigma},
\label{V1est}
\Ee
for all $f\in\Span\sH$ with $\|f\|_{F^s_{p,q}}\leq1$.

We shall next formulate a stronger multiplier result
which involves the Wiener space notion of sequences of bounded $u$-variation.
We recall how these are defined. If $u\ge 1$, we let  $\cV_u(m)$
 be the $u$-variation of the sequence $\{m(j)\}_{j\geq0}$,
defined by
\[\cV_u(m) =\sup \Big(\sum_{n=1}^{N}  |m(j_n)- m(j_{n-1})|^u\Big)^{1/u}\]
with the supremum taken  over all finite strings of numbers 
$\{j_0,\dots,  j_N\}$ satisfying $j_{n-1}<j_{n}$ for $1\le n\le N$, and $j_n\in \bbN\cup\{0\}$.
Note that if $u=1$ we simply have
\[
\cV_1(m)=\sum_{j=1}^\infty|m(j)-m(j-1)|.
\]
We denote by $V_u$ the  space of all $m: \bbN\cup\{0\}\to \SC$ for which
\[
\|m\|_{V_u} := \|m\|_\infty+ \cV_u(m)<\infty.
\]
In particular, if $1\leq u_1\leq u_2<\infty$, it holds 
\[
V_1\hookrightarrow  V_{u_1}\hookrightarrow V_{u_2}\hookrightarrow\ell^\infty, \quad 
\]
As an example, observe that $m(n)=1/(n+1)^{\alpha}$ belongs to $V_1$ for all $\alpha>0$, 
while the alternate sequence $M(n)=(-1)^n m(n)$ belongs to $V_u$ iff $\alpha>1/u$.

We wish to find, in the region of exponents $(s,p,q)$ where $\sH$ is a conditional basis of $F^s_{p,q}$, the largest possible $u$ 
for which $m\in V_u$ implies the boundedness of the operator $T_m$ in $F^s_{p,q}$.
The examples given in \cite{su}, based on multipliers taking the values $0$ and $1$ 
(suitable characteristic functions of finite sets of integers),
 show that for $ 1/u<s-1/q$ there are $m\in V_u$ such that 
the corresponding operators $T_m$ are unbounded on $F^s_{p,q}$; 
see also \S\ref{SS_NC} below.
Our main result in this note  shows that, in the case $1<p,q<\infty$, 
boundedness holds in the complementary range, 
except perhaps at the endpoint.

\begin{theorem}\label{thm:multipliers}
Let
$1<p<q<\infty$ and $1/q\le s<1/p$. Then 
\[\|T_m \|_{F^s_{p,q}\to F^s_{p,q} }
+\|T_m\|_{F^{-s}_{p'q'} \to F^{-s}_{p'q'} } \le C \|m\|_{V_u} , \quad\mbox{if}\quad \frac1u> s-\frac1q.\]
\end{theorem}

\begin{remark}
The appearance of the variation norms is inspired by a result of Coifman, Rubio de Francia and Semmes \cite{CoFrSe88} on Fourier multipliers  
(which is based on the square function result of Rubio de Francia \cite{RubiodeFrancia}, see also \cite{Sjolin86rdf, lacey-rdf}).  
However the variation spaces come up  in quite different ways in  \cite{CoFrSe88} where the variation norm is taken over dyadic intervals 
$[2^j, 2^{j+1})$, with a bound  uniformly  in $j$. This has no analogue in our situation as for each interval $[2^j, 2^{j+1})$ there is 
only one Haar frequency; instead our conditions involve the  variation norms in the parameter $j$.
\end{remark}

\section {Subspaces of $V_u$}

As in \cite{CoFrSe88}, in order to analyze functions in $V_u$ it is convenient to  consider certain subspaces $R_u$ of $V_u$ built on convex  combinations of characteristic functions of unions of disjoint dyadic intervals. This is sketched in \cite{CoFrSe88}, but for the convenience of the reader we give a detailed exposition in the setting of variation spaces for functions on the integers.

For $1\le u<\infty$, let $r_u$ be the class of functions $g:\bbN_0\to \SC$ which are of the form
\[
g=\sum_\nu a_\nu \chi_{I_\nu}, \quad \mbox{with}\quad (\sum_\nu |a_\nu|^u)^{1/u}\le 1,
\]
where the $I_\nu$ are mutually disjoint intervals.
Then $R_u$ is the space of all sequences of the form 
\Be
m=\sum_l c_l g_l, \quad \mbox{with}\quad g_l\in r_u \mand
\sum|c_l|<\infty.
\label{Ru}
\Ee
The norm $\|m\|_{R_u}$ is defined as the infimum of $\sum_l|c_l |$  over all representations as in \eqref{Ru}.
These definitions (for functions on the real line) can be found in  \cite{CoFrSe88}. 
The following result is a discrete analogue of \cite[Lemme 2]{CoFrSe88}, whose proof is sketched for completeness.

\begin{proposition}\label{P_crs} For $\eps>0$, and $1\le u<\infty$ we have 
 \Be\label{incl}R_u\subset  V_u \subset R_{u+\eps}\Ee with continuous embedding.
\end{proposition}

\begin{proof}   
Consider first $g\in r_u$, with 
$g=\sum_\nu a_\nu \chi_{I_\nu}$. 
It is straightforward to see that $\cV_u(g)\le 2\|a\|_{\ell^u}$, thus  $R_u\subset V_u$.

For the second inclusion assume that $f\in V_u$, with $$\cV_u(f)=1,$$ for some $u<\infty$. This implies that $\lim_{n\to \infty} f(n)$ exists and is finite.

Let $w(0)=0$, and for $n\ge 1$, let $w(n)$ be the $u$-th power of the $u$-variation of $f$ over $[0,n]$, that is
\[w(n)
=\sup_{0\le n_0<n_1<\dots< n_N\le n} \sum_{i=1}^N |f(n_i)-f(n_{i-1})|^u.
\] 
Clearly  $w$  is positive, increasing and 
$\cV_u(f)=\lim_{n\to\infty} [w(n)]^{1/u}$ so that $w$
takes values in $[0,1]$. There are two situations (i) $w(n)<1$ for all $n\in \bbN$, and (ii) $w(n)=1$ for $n\ge N_0$ (and some $N_0$). In what follows we assume (i) and omit  the minor modification for (ii)
(in the second case one works with finite sequences instead of infinite sequences).

As stated in \cite[Theorem 2]{bruneau} and used in \cite{CoFrSe88}  one can express  
\[
f=\rho\circ w,\quad \mbox{where 
$\rho\in C^{1/u} [0,1]$}.
\]
 To verify this we may choose  a strictly increasing  sequence of 
nonnegative integers $\{j_n\}_{n=0}^\infty$ so that $j_0=0$ and
 \[
w(j_n)< w(j_{n+1}),\quad w(j_n)=w(k) \quad \mbox{for $j_n\le k<j_{n+1}$}.
\] 
This implies $f(k)=f(j_n)$ for $j_n\le k<j_{n+1}$. 
We now define a piecewise linear function $\rho(t)$, $0\leq t<1$, as follows
\[
\rho(t)= f(j_n)+ \frac{f(j_{n+1})-f(j_n)}{w(j_{n+1})-w(j_n)} (t-w(j_n)),
 \quad w(j_n)\le t<w(j_{n+1}).
\]
So $\rho(0)=f(0)$, and if we let
$\rho(1)= \lim_{n\to \infty} f(n)$,
then $\rho$  is continuous in $[0,1]$. Observe also that $\rho\circ w=f$, since for 
$j_n\le k<j_{n+1}$ we have 
\[\rho(w(k))=\rho(w({j_n}))= f(j_n)=f(k).\]

We now show the  H\"older condition 
\Be 
 \label{Hoelderrho}
|\rho(t)-\rho(t')|\le 3|t-t'|^{1/u}\,, \quad 0\le t'<t\le 1.\Ee
\begin{subequations}
By the definitions of $\rho$ and $w$
we have for $n'<n$
\Be|\rho(w(j_{n}))-\rho(w(j_{n'}))|= |f(j_n)-f(j_{n'})|\le
(w(j_n)-w(j_{n'}))^{1/u}.
\Ee
For $w(j_n)\le t< w(j_{n+1})$
\begin{align}
&|\rho(t)-\rho(w(j_n))| 
 = \Big|\frac{f(j_{n+1})-f(j_n)}{w(j_{n+1})-w(j_n)} \Big| 
 |t-w(j_n)|
\\
&\le \big|w(j_{n+1})-w(j_n)\big |^{-1+\frac 1u} \,  |t-w(j_n)|
\le |t-w(j_n)|^{1/u},
\notag
\end{align}
since $u\geq1$. Similarly
\begin{align}
& |\rho(w(j_{n+1}))-\rho(t)|
= \Big|\frac{f(j_{n+1})-f(j_n)}{w(j_{n+1})-w(j_n)} \Big| 
 |w(j_{n+1})-t|
\\
&\le \big|w(j_{n+1})-w(j_n)\big |^{-1+\frac 1u} |w(j_{n+1})-t|
\le |w(j_{n+1})-t|^{1/u}.
\notag
\end{align}
\end{subequations}
Combining the three cases  we obtain \eqref{Hoelderrho}
for all $t,t'\in [0,1)$, and by continuity the result also holds true on $[0,1]$.

\medskip

We now use the expansion of $\rho$ in terms of the  Haar system in $[0,1]$, that is
$\{\bbone_{[0,1)}, h_{j,\mu}\}$ with $j\geq0$ and $0\leq \mu <2^{j}$. 
Here 
\[
h_{j,\mu} =
\bbone_{I_{j,\mu}^{\text{left}}}-
\bbone_{I_{j,\mu}^{\text{right}}}\]
with
$I_{j,\mu}^{\text{left}}$  and $I_{j,\mu}^{\text{right}}$
 the left and   right halves of $I_{j,\mu}=[2^{-j}\mu, 2^{-j}(\mu+1))$. 
Then \[\rho(t)- \int_0^1\rho(s) ds= \sum_{j=0}^\infty \rho_j(t)\] where
\[\rho_j(t)= \sum_{\mu=0}^{2^j-1} 2^j \inn{h_{j,\mu} }{\rho} h_{j,\mu}(t)
= \rho_{j,1}(t)-\rho_{j,2}(t)\]
with
\[
\rho_{j,1}(t)=\sum_{\mu=0}^{2^j-1} 2^j \inn{h_{j,\mu} }{\rho} 
\bbone_{I_{j,\mu}^{\text{left}}}(t),\mand 
\rho_{j,2}(t)=\sum_{\mu=0}^{2^j-1} 2^j \inn{h_{j,\mu} }{\rho} \bbone_{I_{j,\mu}^{\text{right}}}(t).
\]
Now \[ |2^j \inn{h_{j,\mu} }{\rho} |
= 2^j\Big| \int h_{j,\mu}(t) \big[\rho(t)-\rho\big(2^{-j}(\mu+\tfrac 12)\big)\big] dt\Big|  \le 3\cdot 2^{-1/u}\,
2^{-j/u},
\]
by \eqref{Hoelderrho}. Thus, if $\e>0$ we have
\[ \Big(\sum_{\mu=0}^{2^j -1} |2^j \inn{h_{j,\mu} }{\rho} |^{u+\eps}\Big)^{\frac{1}{u+\eps}}\le 
3\cdot 2^{-\frac1u}\, 2^{-\frac ju}\, 2^{\frac j{u+\e}}=: c_{j,\e}.
\]
Since $w$ is increasing 
it is clear that the functions 
$$n\mapsto \bbone_{I_{j,\mu}^{\text{left}}}\big(w(n)\big), \quad
n\mapsto \bbone_{I_{j,\mu}^{\text{right}}}\big(w(n)\big)$$
are characteristic functions of intervals restricted to the integers. For fixed $j$ these intervals are also mutually disjoint,
so we see that
\[
g_{j,\e}:=\frac1{2c_{j,\e}}\,\big(\rho_j\circ w\big)\in r_{u+\e}.
\]
Since $C_\e:=2\sum_{j\geq0} |c_{j,\e}|<\infty$, it then follows that
\[
f=\rho\circ w=\int_0^1\rho+ \sum_{j=0}^\infty2c_{j,\e}\,g_{j,\e}\in R_{u+\eps},\quad  \mbox{with $\|f\|_{R_{u+\eps}}\le 1+C_\eps$}.
\]
\end{proof} 

\begin{remark}\label{R_Rus}
Observe that the previous proof actually shows that, if $m\in V_u$ and $\e>0$, then one can write $m=\sum_{j=0}^\infty c_j m_j$ 
with $m_j\in r_{u+\e}$ and $\sum_{j=1}^\infty|c_j|^\sigma<\infty$, for all $\sigma>0$. 
\end{remark}

\section{\it The proof of Theorem \ref{thm:multipliers}}

We shall actually prove a stronger result than Theorem \ref{thm:multipliers}, 
which provides optimal boundedness (up to endpoints) in a slightly larger region of indices. 
Given a fixed $q>1$, we denote by $\T_q$ the open triangle in the plane $(1/p,s)$ with 
vertices $(1,1), (1/q,1/q)$ and $(1+1/q,1/q)$; see Figure \ref{fig_Tq}. 

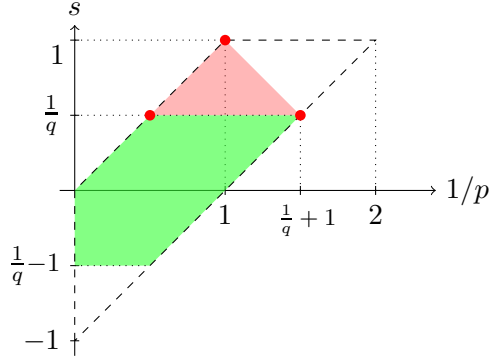
\begin{figure}[h]
 \centering
{\begin{tikzpicture}[scale=2]


\draw[->] (-0.1,0.0) -- (2.4,0.0) node[right] {${1}/{p}$};
\draw[->] (0.0,-0.0) -- (0.0,1.1) node[above] {$s$};
\draw (0.0,-1.1) -- (0.0,-1.0)  ;

\draw (1.0,0.03) -- (1.0,-0.03) node [below] {$1$};
\draw (2.0,0.03) -- (2.0,-0.03) node [below] {$2$};
\draw (1.5,0.03) -- (1.5,-0.03) node [below] {{\footnotesize $\;\;\frac1q+1$}};
\draw (0.03,1.0) -- (-0.03,1.00);
\node [left] at (0,0.9) {$1$};
\draw (0.03,.5) -- (-0.03,.5) node [left] {$\tfrac{1}{q}$};
\draw (0.03,-.5) -- (-0.03,-.5) node [left] {$\tfrac{1}{q}${\small{$-1$}}};
\draw (0.03,-1.0) -- (-0.03,-1.00) node [left] {$-1$};

\draw[dotted] (1.0,0.0) -- (1.0,1.0);
\draw[dotted] (0,1.0) -- (1.0,1.0);
\draw[dotted] (1.5,0.0) -- (1.5,0.5);
\draw[dotted] (2,0.0) -- (2,1.0);

\path[fill=red!70, opacity=0.4] 
(.5,.5)-- (1.5,0.5) -- (1,1);

\path[fill=green!80, dashed, opacity=0.6] (0.0,0.0) -- (.5,.5)-- (1.5,0.5) -- (1,0)--(.5,-.5) -- (0,-0.5)--(0,0);
\draw[dotted] (0,0.5)--(1.5,0.5);
\draw[dotted] (0,-0.5)--(0.5,-0.5);

\draw[dashed] (0.0,-1.0) -- (0.0,0.0) -- (1.0,1.0) -- (2,1.0) -- (1.0,0.0) --
(0.0,-1.0);

\fill[red] (1,1) circle (1pt);
\fill[red] (0.5,0.5) circle (1pt);
\fill[red] (1.5,0.5) circle (1pt);

\end{tikzpicture}
}

\caption{In red, the region $\T_q$. 
}\label{fig_Tq}
\end{figure}

For this region we give  the following result, which includes Theorem \ref{thm:multipliers} as a special case.

\begin{theorem}
\label{th_Tm}
Let $1<q<\infty$ and $(1/p,s)\in \T_q$. Then, for all $u\geq 1$ such that $1/u>s-1/q$, it holds
\Be
\|T_{m} f\|_{F^s_{p,q}\to F^s_{p,q}} \leq\, c\, \|m\|_{V_u},\quad m\in V_u .
\label{Tm}
\Ee
Moreover, a necessary condition for \eqref{Tm} to hold for all such $m$ is that $1/u\geq s-1/q$.
\end{theorem}

In view of Proposition \ref{P_crs}, we shall first consider sequences from the class
$r_u$, that is multipliers $m$ of the form 
\Be
m[a,\cI]=\sum_\nu a_\nu \bbone_{I_\nu},
\label{ma}
\Ee
where  
$\cI=\{I_\nu\}$ is a family of disjoint intervals and $a=\{a_\nu\}$ a sequence in $\ell^u$.
For these multipliers we have the following result.

\begin{proposition}
\label{P_Tma}
Let $1<q<\infty$ and $(1/p,s)\in \T_q$. Then, for all $u\geq 1$ such that $1/u>s-1/q$, there exists $c=c(p,q,s,u)>0$ such that 
\Be
\|T_{m[a,\cI]} f\|_{F^s_{p,q}} \leq\, c\, \|a\|_{\ell^u} \|f\|_{F^s_{p,q}}, \quad \forall\;f\in F^s_{p,q}(\SR), \;a\in\ell^u,
\label{Tma}
\Ee
for every multiplier $m[a,\cI]$ defined as in \eqref{ma}. 
Moreover, a necessary condition for \eqref{Tma} to hold for all such $m[a,\cI]$ is that $1/u\geq s-1/q$.
\end{proposition}
\begin{remark}
We emphasize that the constant $c$ in \eqref{Tma} 
does not depend on the family of disjoint intervals $\cI=\{I_\nu\}$.
\end{remark}

In the next subsections we shall prove Proposition \ref{P_Tma}. For the sufficiency part we shall use complex interpolation
applied to the bilinear operator 
\[
(a,f)\longmapsto \cT[a,f]:= T_{m[a,\cI]} f.
\] 
For simplicity we shall remove the dependence on $\cI$ in the subsequent notation, as it will be clear
from the proofs that the involved constants do not depend on it.

\subsection{\it Interpolation with varying $q$ and $p=1$}\label{qvary}
We first prove an  inequality for $p=1$ which is  efficient for  $s$  near $1$, namely
\Be \label{u-one-claim}      \|\cT[a,f]\|_{F^s_{1,q}} \lc \|a\|_{\ell^u}  \|f\|_{F^s_{1,q}}, \quad s<1, \quad
1/u>1-1/q.
\Ee 
Since the Haar system is an unconditional basis on $F^s_{1,1}=B^s_{1,1}$, $0<s<1$ (see  Theorem 2.9 in \cite{triebel-bases}) we have 
\[ \|\cT[a,f]\|_{F^s_{1,1}} \lc \|a\|_{\ell^\infty}  \|f\|_{F^s_{1,1}}, \quad 0<s<1.
\]
Next, the uniform boundedness of the operators $\SE_N$ in $F^s_{1,q_1}$ (see \cite[Corollary 1.3]{gsu}) and the trivial estimate in 
\eqref{V1est} (with $\sigma=1$) imply that for any $q_1\in (q,\infty)$  
\[\|\cT[a,f]\|_{F^s_{1,q_1}} \lc \|a\|_{1}  \|f\|_{F^s_{1,q_1}}, \quad 0<s<1.
\]
By complex interpolation we then obtain \eqref{u-one-claim} for $1/u= (1-1/q)/ (1-1/q_1)$, 
which after choosing $q_1$ large enough implies \eqref{u-one-claim} whenever $1/u>1-1/q$.

\subsection{\it Interpolation with  fixed $q$}
Let $q\in(1,\infty)$ be fixed, and let $(1/p,s)\in \T_q$.
We shall prove \eqref{Tma} by interpolating sufficiently close to the upper vertex $(1, 1)$ and the lower segment $(1/p_1, 1/q)$ of $\T_q$.

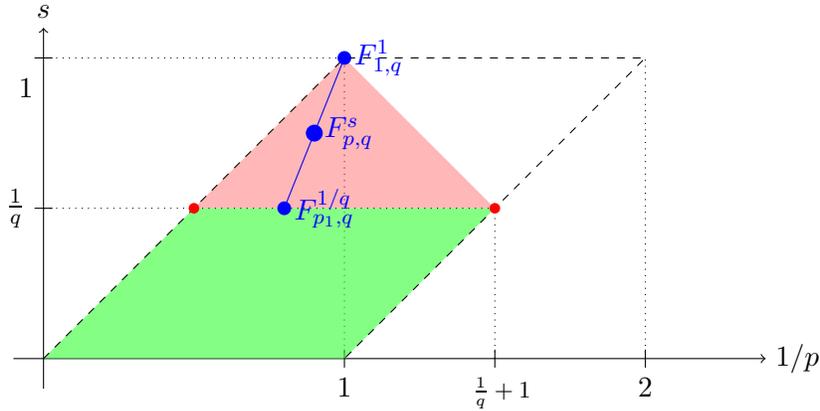
\begin{figure}[h]
 \centering
{\begin{tikzpicture}[scale=4]


\draw[->] (-0.1,0.0) -- (2.4,0.0) node[right] {${1}/{p}$};
\draw[->] (0.0,-0.1) -- (0.0,1.1) node[above] {$s$};

\draw (1.0,0.03) -- (1.0,-0.03) node [below] {$1$};
\draw (2.0,0.03) -- (2.0,-0.03) node [below] {$2$};
\draw (1.5,0.03) -- (1.5,-0.03) node [below] {{\footnotesize $\;\;\frac1q+1$}};
\draw (0.03,1.0) -- (-0.03,1.00);
\node [left] at (0,0.9) {$1$};
\draw (0.03,.5) -- (-0.03,.5) node [left] {$\tfrac{1}{q}$};

\draw[dotted] (1.0,0.0) -- (1.0,1.0);
\draw[dotted] (0,1.0) -- (1.0,1.0);
\draw[dotted] (1.5,0.0) -- (1.5,0.5);
\draw[dotted] (2,0.0) -- (2,1.0);

\path[fill=red!70, opacity=0.4] 
(.5,.5)-- (1.5,0.5) -- (1,1);

\path[fill=green!80, dashed, opacity=0.6] (0.0,0.0) -- (.5,.5)-- (1.5,0.5) -- (1,0)--(0,0);
\draw[dotted] (0,0.5)--(1.5,0.5);

\draw[dashed] (0.0,0.0) -- (1.0,1.0) -- (2,1.0) -- (1.0,0.0);

\fill[blue] (1,1) circle (0.65pt) node [right] {$F^1_{1,q}$};
\fill[red] (0.5,0.5) circle (0.5pt);
\fill[red] (1.5,0.5) circle (0.5pt);

\fill[blue] (0.8,0.5) circle (0.65pt) node [right] {$F^{1/q}_{p_1,q}$};
\draw[blue] (0.8,0.5)--(1,1);
\fill[blue] (0.9,0.75) circle (0.8pt) node [right] {$F^s_{p,q}$};

\end{tikzpicture}
}

\caption{Interpolation strategy for points $(1/p,s)\in\T_q$.
}\label{fig_interpol1}
\end{figure}


Let $P=(1/p,s)\in\T_q$ and $\e_1>\e_0>0$ be sufficiently small, to be chosen. Let $P_0=(1, s_0)$ with $s_0=1-\e_0$. 
Draw a line through $P_0$ and $P$, and let $P_1=(1/p_1, s_1)$ be the intersection with the horizontal line $s_1=1/q-\e_1$.
That is,
\[
P=(1-\theta)\cdot P_0+\theta\cdot P_1, \quad \mbox{with} \quad \theta=\frac{s_0-s}{s_0-s_1}.
\]
Choosing $\e_0,\e_1$ sufficiently small we can guarantee that $s_1<s<s_0$ (and hence $\theta\in(0,1)$), 
and that $P_1$ lies in the green region.
Next, take
\[
\frac1{u_0}:=1-\frac1q\,+\e_1-\e_0 >1-\frac1q.
\] 
From the previous step and the unconditional basis property we have
\[
\cT:\ell^{u_0}\times F^{s_0}_{1,q}\to F^{s_0}_{1,q}\mand 
\cT:\ell^{\infty}\times F^{s_1}_{p_1,q}\to F^{s_1}_{p_1,q}.
\]
Using complex interpolation this yields
\Be
\|\cT[a,f]\|_{F^{s}_{p,q}} \lc \|a\|_{\ell^{u}}  \|f\|_{F^{s}_{1,q}},
\label{Tuaf}
\Ee
with 
\[
\frac 1u=\frac{1-\theta}{u_0}+\frac\theta\infty=\frac{s-s_1}{s_0-s_1}\cdot\frac1{u_0}=s-\frac1q+\e_1.
\]
Letting $\e_1\searrow0$ we deduce the validity of \eqref{Tuaf} whenever $\frac1u>s-\frac1q$.
This completes the proof of the sufficient condition in Proposition \ref{P_Tma}.

\subsection{\it Necessary condition}\label{SS_NC}

Suppose first that $1<p<q$ with $1/q<s<1/p$. Then, the example constructed in \cite[\S5]{su} gives a multiplier of the form $m=\bbone_E$, so that $\card{\,E}=2^N$ (with the elements in $E$ being $N$-separated), and with the property that
\Be
\|T_m\|_{F^s_{p,q}\to F^s_{p,q}}\gtrsim 2^{N(s-\frac1q)}.
\label{Tmspq}
\Ee
Since we can write $m$ in the form \eqref{ma} (with $I_\nu=\{\nu\}$ and $a_\nu=1$, for $\nu\in E$), then 
the validity of \eqref{Tma} will imply that
\[
2^{N(s-\frac1q)} \lesssim \|T_{m}\|_{F^s_{p,q}\to F^s_{p,q}}\lesssim \|a\|_{\ell^u}=2^{N/u}.
\]
Thus, we must necessarily have $1/u\geq s-1/q$.

Arguing by interpolation as in \cite[\S7]{su} one can show that \eqref{Tmspq} 
(with an $\e$ loss) continues to hold for all $(1/p,s)$ with 
\Be\max\{1/q, 1/p-1\}<s<\min\{1/p,1\}
\label{Quad}
\Ee
which is a larger region than $\T_q$; see Figure \ref{fig_Tq}.

To be more precise, let $P_1=(1/p,s)$ belong to the open quadrilateral defined by \eqref{Quad}, where we assume $p\leq 1$. We shall interpolate close to the points shown in Figure \ref{fig3}.

\begin{figure}[h]
 \centering
{\begin{tikzpicture}[scale=3]


\draw[->] (-0.1,0.0) -- (2.4,0.0) node[right] {${1}/{p}$};
\draw[->] (0.0,-0.1) -- (0.0,1.1) node[above] {$s$};

\draw (1.0,0.03) -- (1.0,-0.03) node [below] {$1$};
\draw (2.0,0.03) -- (2.0,-0.03) node [below] {$2$};
\draw (1.5,0.03) -- (1.5,-0.03) node [below] {{\footnotesize $\;\;\frac1q+1$}};
\draw (0.03,1.0) -- (-0.03,1.00);
\node [left] at (0,0.9) {$1$};
\draw (0.03,.5) -- (-0.03,.5) node [left] {$\tfrac{1}{q}$};

\draw[dotted] (1.0,0.0) -- (1.0,1.0);
\draw[dotted] (0,1.0) -- (1.0,1.0);
\draw[dotted] (1.5,0.0) -- (1.5,0.5);
\draw[dotted] (2,0.0) -- (2,1.0);

\path[fill=red!70, opacity=0.4] 
(.5,.5)-- (1.5,0.5) -- (2,1)--(1,1);

\path[fill=green!80, dashed, opacity=0.6] (0.0,0.0) -- (.5,.5)-- (1.5,0.5) -- (1,0)--(0,0);
\draw[dotted] (0,0.5)--(1.5,0.5);

\draw[dashed] (0.0,0.0) -- (1.0,1.0) -- (2,1.0) -- (1.0,0.0);

\fill[blue] (1.5,0.85) circle (0.65pt) node [right] {$F^s_{p,q}$};

\fill[blue] (0.5,0.5) circle (0.65pt) node [left] {$F^{1/q}_{q,q}$};
\draw[blue] (1.5,0.85)--(.5,.5);
\fill[blue] (1,0.675) circle (0.8pt) node [above] {$F^{s_\theta}_{1,q}$};

\end{tikzpicture}
}
\caption{}
\label{fig3}
\end{figure}
Namely, given $\e>0$, let $P_0=(\frac1q,\frac1q-\e)$. Draw a segment from $P_0$ to $P_1$, and consider the convex combination of $P_0$ and $P_1$ with first coordinate $(1+\eps)^{-1}$; i.e 
let $\theta\in(0,1)$ and $s_\theta$ be such that
\Be
\big(\tfrac 1{1+\eps}, s_\theta\big)=(1-\theta)\,\big(\tfrac 1q, \tfrac 1q-\eps\big) \,+\,\theta\,\big(\tfrac 1p,s\big) .
\label{Ptheta}
\Ee
Then, by complex interpolation we have
\[
\|T_m\|_{F^{s_\theta}_{1+\e, q}}\,\lesssim\, \|T_m\|^{1-\theta}_{F^{\frac1q-\e}_{q,q}}\,\|T_m\|^\theta_{F^s_{p,q}}
\]
By unconditionality, $\|T_m\|_{F^{\frac1q-\e}_{q,q}}\lesssim 1$, so we arrive at
\[
\|T_m\|_{F^s_{p,q}} \,\gtrsim\,
\|T_m\|_{F^{s_\theta}_{1+\e, q}}^{1/\theta} \,\gtrsim \,2^{\frac N\theta\,(s_\theta-\frac1{q})},
\]
the last bound due to \eqref{Tmspq}.
Now, solving for $s_\theta$ in \eqref{Ptheta} we see that
\[
s_\theta-\frac1q=\big(s-\frac1q\big)\theta -(1-\theta)\e.
\]
Thus, 
\[
\|T_m\|_{F^s_{p,q}} \,\gtrsim\,
\,2^{ N\,[(s-\frac 1q)-\frac{1-\theta}\theta\,\e]}.
\]
So, if \eqref{Tma} was true, arguing as above we would arrive at
\[
\frac1u\geq (s-\frac 1q)-\frac{1-\theta}\theta\,\e,
\]
which letting $\e\searrow 0$ leads to $\frac1u\geq s-\frac 1q$.

\subsection{\it Conclusion of the proof of Theorem \ref{th_Tm}}
Let $q>1$ and let 
$(1/p,s)\in\T_q$ be fixed. Let $1/u>s-1/q$ and $m\in V_u$ with $u\geq1$. Then, for some $u_1>u$ we also have $1/u_1>s-1/q$. By  Remark \ref{R_Rus} we can write 
$m=\sum_{j=0}^\infty c_j m_j$ with $m_j\in r_{u_1}$ and $\sum_{j=0}^\infty |c_j|^\sigma\lesssim \|m\|^\sigma_{V_u}$, with $\sigma=\min\{1,p\}$. Then, using the $\sigma$-triangle inequality, we have
\[
\|T_mf\|_{F^s_{p,q}}^\sigma\leq \sum_{j=0}^\infty |c_j|^\sigma\,\|T_{m_j}f\|_{F^s_{p,q}}^\sigma, \quad f\in F^s_{p,q}.
\]
By Proposition \ref{P_Tma}, $\|T_{m_j}f\|_{F^s_{p,q}}\lesssim \|f\|_{F^s_{p,q}}$, for all $j\geq 0$, so we conclude that 
\[
\|T_m\|_{F^s_{p,q}\to F^s_{p,q}} \lesssim \|m\|_{V_u}.
\]
\qed
\remark When $p>1$, the assertion $$ \|T_m\|_{F^{-s}_{p',q'}\to F^{-s}_{p',q'}} \lesssim \|m\|_{V_u}$$ stated in Theorem \ref{thm:multipliers} follows from \eqref{Tm} by duality.
So, when $q>1$, the condition on $V_u$ is optimal (up to endpoints) also in the lower triangle
on the left of Figure \ref{fig1}. 

\remark When $1/2<p\leq1$, we did not state any result for the right upper triangle in Figure \ref{fig_Tq}. It is also possible to obtain, by complex interpolation, a sufficient condition for multipliers of the form $m[a,\cI]$
in terms of $\|a\|_{\ell^u}$, although in this range the value of $u$ will no longer match the necessary condition from \S\ref{SS_NC}.

\remark When $1/2<q\leq1$, one can also prove by interpolation, for multipliers of the form $m[a,\cI]$, that \[
\|a\|_{\ell^u}<\infty, \quad \frac1u>\frac1q-1-s,\]
is a sufficient condition in the open triangle with vertices $(0,-1)$, $(1/q,1/q-1)$ and $(1/q-1,1/q-1)$; see Figure \ref{fig4} below. This matches the necessary condition from the examples in \cite[\S5]{su} (except for the endpoint). In the remaining part of the figure, however, the sufficient condition obtained by interpolation will be weaker than this one.

\begin{figure}[h]
 \centering
 {
\begin{tikzpicture}[scale=2]


\draw[->] (-0.1,0.0) -- (2.3,0.0) node[right] {${1}/{p}$};
\draw[->] (0.0,-0.0) -- (0.0,1.1) node[above] {$s$};
\draw (0.0,-1.1) -- (0.0,-1.0)  ;

\draw (1.0,0.03) -- (1.0,-0.03) node [below] {$1$};
\draw (2,0.03) -- (2,-0.03) node [below] {$2$};
\draw (0.03,1.0) -- (-0.03,1.00);
\node [left] at (0,0.9) {$1$};
\draw (0.03,.5) -- (-0.03,.5) node [left] {$\tfrac{1}{q}$ {\footnotesize $-1$}};
\draw (0.03,-1.0) -- (-0.03,-1.00) node [left] {$-1$};

\draw[dotted] (1.0,0.0) -- (1.0,1.0);
\draw[dotted] (0,1.0) -- (1.0,1.0);
\draw[dotted] (2,0.0) -- (2,1.0);

\path[fill=green!70, opacity=0.4] (.5,.5)-- (1,1)--(2,1) -- (1.5,0.5)--(0.5,0.5);
\draw[dotted] (0,0.5)--(1.5,0.5);

\draw[dashed] (0.0,-1.0) -- (0.0,0.0) -- (1.0,1.0) -- (2,1.0) -- (1.0,0.0) --
(0.0,-1.0);

\fill[red] (0.5,0.5) circle (1pt) ;
\fill[red] (1.5,0.5) circle (1pt) ;
\fill[red] (0,-1) circle (1pt) ;
\path[fill=red!70, opacity=0.4] (.5,.5)--(1.5,0.5)--(0,-1);

\end{tikzpicture}
}
\caption{Parameter domain for the cases $1/2<q<1$.}\label{fig4}
\end{figure}
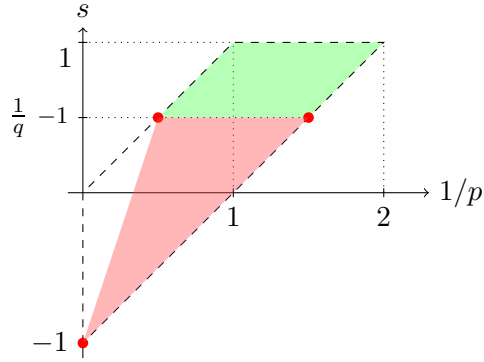

\remark It may be interesting to note that even for the special case of Sobolev spaces
$H^s_p = F^s_{p,2}$, $1<p<2$, $1/2\le s<1/p$ and $u<\frac{2}{2s-1}$  the use of 
 Triebel-Lizorkin spaces $F^s_{p,q}$ with $q\neq 2$ is crucial. 
Such interpolation arguments were used for   multiplier transformations   in other contexts  to establish endpoint results on
  Lorentz spaces $L_{p,2}$, see  \cite{SeeMonatshefte, See89} for basic versions, 
	and \cite{SeeTao01, TaoWr01} for  more advanced versions. 


\begin{thebibliography}{10}



\bibitem{bruneau} M. Bruneau, \emph{Variation totale d'une fonction.} (French) Lecture Notes in Mathematics, Vol. 413. Springer-Verlag, Berlin-New York, 1974. 


\bibitem{CoFrSe88} R. 
Coifman, J.-L.  Rubio de Francia, S.  Semmes. \emph{ Multiplicateurs de Fourier de $L^p(\bbR)$  et estimations quadratiques.}  C. R. Acad. Sci. Paris S\'er. I Math. 306 (1988), no. 8, 351--354. 





\bibitem{gsu} Gustavo  Garrig\'os, Andreas Seeger, T. Ullrich. \emph{The Haar system as a Schauder basis in spaces of Hardy-Sobolev type.}
 Jour. Fourier Anal.  Appl., 24 (5) (2018), 1319--1339.

\bibitem{gsu-endpt}  \bysame.
\emph{Basis properties of the Haar system in limiting Besov spaces.} In \emph{Geometric Aspects of Harmonic Analysis}, 
P. Ciatti, A. Martini (eds), Springer Indam Series 45 (2021), pp 361--424. Also available as preprint arxiv:1901.09117.

\bibitem{gsu-TL} \bysame. \emph{The Haar system in Triebel-Lizorkin spaces: Endpoint results.}
J. Geom. Anal. 31 (2021), no. 9, 9045--9089.




\bibitem{lacey-rdf} 
Michael T.  Lacey. \emph{
Issues related to Rubio de Francia's Littlewood-Paley inequality.} 
New York Journal of Mathematics. NYJM Monographs, 2. State University of New York, University at Albany, Albany, NY, 2007. 36 pp.

\bibitem{Pe} Jaak Peetre. \emph{On spaces of Triebel-Lizorkin type.}  Ark. Mat. 13 (1975),123--130.
%

\bibitem{RubiodeFrancia} J.L. Rubio de Francia. \emph{A {L}ittlewood-{P}aley inequality for arbitrary intervals.}
   Rev. Mat. Iberoamericana {\bf 1} (2) (1985), 1--14.
     
\bibitem{RuSi96}Thomas  Runst, Winfried Sickel. \newblock \emph{Sobolev spaces of fractional order, {N}emytskij operators, and  nonlinear partial differential equations}, volume~3 of  de Gruyter Series  in Nonlinear Analysis and Applications. \newblock Walter de Gruyter \& Co., Berlin, 1996.

\bibitem{SeeMonatshefte} Andreas  Seeger. {A limit case of the {H}\"{o}rmander multiplier theorem}.
Monatsh. Math. {\bf 105} (2) (1988), 151--160.


\bibitem{See89} \bysame. \emph{Estimates near $ L^1$  for Fourier multipliers and maximal functions.} 
Arch. Math. (Basel) 53 (1989), no. 2, 188--193. 


\bibitem{SeeTao01} Andreas Seeger, Terence Tao. \emph{
Sharp Lorentz space estimates for rough operators.} Math. Ann. 320 (2001), no. 2, 381--415.

\bibitem{su} Andreas Seeger, Tino Ullrich. \emph{Haar projection numbers and failure of unconditional convergence in Sobolev spaces.} Math. Z. 285 (2017), 91 -- 119. 


\bibitem{sudet} \bysame. \emph{Lower bounds for Haar projections: Deterministic Examples.} Constr. Appr. 46 (2017), 227--242. 


\bibitem{Sjolin86rdf} Per  Sj\"{o}lin.
     \emph{A note on {L}ittlewood-{P}aley decompositions with arbitrary
              intervals}. J. Approx. Theory {\bf 48} (3) (1986), 328--334.
     

\bibitem{TaoWr01} Terence Tao,  James Wright. \emph{  Endpoint multiplier theorems of Marcinkiewicz type.} Rev. Mat. Iberoamericana 17 (2001), no. 3, 521--558.




\bibitem{Tr83} H. Triebel.\newblock \emph{Theory of function spaces}. Birkh\"auser Verlag, Basel, 1983.



\bibitem{triebel2} \bysame. \emph{Theory of function spaces II.} Monographs in Mathematics, 84. Birkh\"auser Verlag, Basel, 1992.



\bibitem{triebel-bases} \bysame. \emph{Bases in function spaces, sampling,discrepancy, numerical integration.} EMS Tracts in Mathematics, 11. European Mathematical Society (EMS), Z\"urich, 2010.



\end{thebibliography}

\newpage
\end{document}